%% file: main.tex
\documentclass[11pt,a4paper]{article}
\usepackage{amssymb}
\usepackage{amsthm}
\usepackage{amsmath}
\usepackage{amsbsy}
\usepackage[a4paper, left=2.0cm, right=2.0cm, top=2.5cm, bottom=2.5cm]{geometry}
\usepackage[T1]{fontenc}
\usepackage{lmodern}
\usepackage{subfig}
\usepackage{graphicx}
\usepackage{caption}
\usepackage{float}
\usepackage[usenames,dvipsnames]{xcolor}
\usepackage[linktocpage,colorlinks=true]{hyperref}
\usepackage{todonotes}

\usepackage{tikz}
\usetikzlibrary{matrix,arrows,calc,shapes}

\makeatletter
\newtheorem*{rep@theorem}{\rep@title}
\newcommand{\newreptheorem}[2]{%
\newenvironment{rep#1}[1]{%
 \def\rep@title{#2 \ref{##1}}%
 \begin{rep@theorem}}%
 {\end{rep@theorem}}}
\makeatother

\theoremstyle{plain}
\newtheorem{theorem}{Theorem}[section]
\newreptheorem{theorem}{Theorem}
\newtheorem{proposition}[theorem]{Proposition}
\newtheorem{lemma}[theorem]{Lemma}
\newreptheorem{lemma}{Lemma}

\newtheorem{corollary}[theorem]{Corollary}

\theoremstyle{definition}
\newtheorem{definition}[theorem]{Definition}
\newtheorem{remark}[theorem]{Remark}
\newtheorem{example}[theorem]{Example}
\theoremstyle{definition}

\DeclareMathOperator{\cv}{cv}

\newcommand{\R}{\mathbb{R}}

\providecommand{\dgm}{\mathsf{Dgm}}

\providecommand{\per}{\mathsf{Per}}

\newcommand{\str}{\mathsf{Str}}

\newcommand{\sD}{{\mathsf D}}
\newcommand{\sF}{{\mathsf F}}
\newcommand{\sN}{{\mathsf N}}
\newcommand{\sP}{{\mathsf P}}

\newcommand{\cA}{{\mathcal{A}}}

\newcommand{\complexx}[1]{{\mathcal{S}_N^{#1}}}
\newcommand{\complex}{{\mathcal{S}_N}}

\newcommand{\setof}[1]{\left\{ {#1}\right\}}

\input{correctm}

\showcomments
\begin{document}

\begin{center}
{\large \bf Contractibility of a persistence map preimage}

 \vskip 0.5cm
{\large
Jacek Cyranka$^{*,\dag}$, Konstantin Mischaikow$^{*}$, and Charles Weibel$^{*}$
%authors
}
\vskip 0.5cm
{\small$^*$ Department of Mathematics, Rutgers, The State University of New Jersey,\\110 Frelinghusen Rd, Piscataway, NJ  08854-8019, USA}\\
\vskip 0.5cm
{\small$^\dag$ Institute of Informatics, University of Warsaw,\\ Banacha 2, 02–097 Warsaw, Poland
}
\vskip 0.5cm
{jcyranka@gmail.com, mischaik@math.rutgers.edu, weibel@math.rutgers.edu}
\vskip 0.5cm

\today
\end{center}

\noindent{\bf Abstract.}
This work is motivated by the following question in data-driven study
of dynamical systems: given a dynamical system that is observed via
time series of persistence diagrams that encode topological features
of snapshots of solutions, what conclusions can be drawn about solutions
of the original dynamical system?  
In this paper we provide a
definition of a persistence diagram for a point in $\R^N$.  
We then provide conditions under which
time series of persistence diagrams can be used to guarantee the
existence of a fixed point of the flow on $\R^N$ that generates the
time series.  
To obtain this result requires an understanding of the
preimage of the persistence map.  
The main theorem of this paper gives conditions under which these preimages are contractible simplicial complexes.


\section{Introduction}
\label{secintro}
\emph{Topological data analysis} (TDA), especially in the form of
persistent homology, is rapidly developing into a widely used tool for
the analysis of high dimensional data associated with nonlinear
structures.  
That topological tools can play a role in this subject
should not be unexpected, given the central role of nonlinear
functional analysis in the study of geometry, analysis, and
differential equations.  
What is perhaps surprising is that, to the
best of our knowledge, there has been no systematic attempts to develop analogous techniques to process information obtained via persistent homology.

Persistent homology is often used as a means of data reduction.  A
typical example takes the form of a complicated scalar function
defined over a fixed domain, where the geometry of the
sub-(super)-level sets is encoded via homology.  Of particular
interest to us are settings in which the scalar function arises as a
solution to a partial differential equation (PDE); we are interested
in tracking the evolution of the function, but experimental data only
provides information on the level of digital images of the process.
Furthermore,  
capturing the dynamics of a PDE often 
requires a long time series of rather 
large digital images.  Thus, rather than storing
the full images, one can hope to work with a time series of
persistence diagrams.  Our aim is to draw conclusions about the
dynamics of the original PDE from the time series of the persistence
diagrams.  This is an extremely ambitious goal and far beyond our
capabilities at the moment.  A much simpler question is the following:
if there is an attracting region in the space of persistence diagrams,
under what conditions can we conclude that there is a fixed point for
the PDE?

This paper represents a first step towards answering the simpler
question.  Theorem~\ref{thm:maindyn} shows that given an ordinary
differential equation (ODE) with a global compact attractor
$\cA\subset\R^N$ and a neighborhood in the space of persistence
diagrams that is mapped into itself under the dynamics, then there
exists a fixed point for the ODE.  In applications one could consider
the ODE as arising from a finite difference approximation of the PDE.

The challenge is that to obtain results one must understand the
topology of $data_P$, the space  of data having a fixed persistence diagram $P$,
a topic for which there are only limited results.  
That the structure of
$data_P$ is complicated follows directly from the fact that
persistent homology can provide tremendous data reduction, but in a
highly nonlinear fashion.  
With this in mind, the primary goal of this paper is to show that for a reasonable class of problems the space $data_P$ is a finite set of contractible, simplicial sets.
The importance of this result is that it opens the possibility of
applying standard algebraic topological tools, e.g., Lefschetz fixed
point theorem, Conley index, to dynamics that is observed through the
lens of persistent homology.

To state our goal precisely requires the introduction of notation.
Throughout this paper $\complex$ denotes the 1-dimensional simplicial 
complex composed out of $N$ vertices $[i]$ ($i=1,...,N$) and 
$N-1$ edges $[i,i+1]$ ($i=1,....,,N-1$). 
It is a simplicial decomposition of closed bounded interval  in $\R$.

We study filtrations of $\complex$ defined as follows. 

\begin{definition}
\label{def:Sfiltration}
Let $z=(z_1,\ldots,z_N)\in \R^N$.
Define $f\colon \R^N\times\complex\to \R$ by
\[
    f(z,\sigma) := \begin{cases}
    z_j &\text{if $\sigma=[j]$,}\\
    \max\setof{z_j, z_{j+1}}&\text{if $\sigma=[j,j+1]$.}
    \end{cases}
\]
For $r\in \R$, we set 
$
\complex(z,r) := \setof{\sigma\in\complex : f(z,\sigma)\leq r}.
$
\end{definition}

\begin{definition} 
\label{defn:sublevelset}
Given $z=(z_1,\ldots,z_N)\in \R^N$, we can 
reorder the coordinates of $z$ such that
\[
z_{j_1} \leq z_{j_2} \leq \cdots \leq z_{j_N}.
\]
The \emph{sublevel-set filtration\footnote{Analogous results can be obtain for superlevel set filtrations (see Section~\ref{sec:conclusion}).}
of $\complex$ at $z$},
which we write as $\complexx{\sF}(z)$, is given by
\[
   \complex(z,z_{j_1}) \subseteq  \complex(z,z_{j_2}) \subseteq \cdots
   \subseteq \complex(z,z_{j_N}).
\]
\end{definition}

Because $\complexx{\sF}(z)$ is a finite filtration of simplicial complexes, completely determined by $z$, we can use classical results from \cite{comptop, ZC} to compute 
the persistence diagram of $\complexx{\sF}(z)$. 
We treat this as a map
\[
\dgm\colon \R^N\to \per,
\]
where $\per$ denotes the space of all persistence diagrams. 
Thus the space $data_P$  of all $z\in\R^N$ having 
persistence diagram $P$ is just $\dgm^{-1}(P)$.
We remark that there are a variety of topologies that can be put on $\per$ such
that $\dgm$ becomes a continuous map \cite{chazal,stability}.

Since $\complex$ is one-dimensional and 
contractible, we are only concerned with the persistent homology $H_0$,
i.e., the persistence diagrams 
associated with connected components.
Therefore for the rest of the paper we restrict our study to 
consist of the family $\per$ of persistence diagrams of level zero.

Here is the main result of this paper. 

\begin{theorem}
\label{thm:main}
For every persistence diagram $P$, the space $data_P\subset \R^N$ 
is composed of a finite number of mutually disjoint
components. Each component is contractible, and is homeomorphic to a finite union of convex, potentially unbounded polytopes.
\end{theorem}

The proof of Theorem~\ref{thm:main} is not particularly difficult, but
it is technical. 
We first describe the connected components of $data_P$; 
see Lemma \ref{lemma:local_extrema}. 
In Section \ref{sec:cellular}, we introduce the poset $\str$ of cellular strings, which are be used to decompose each component as a finite union of convex polytopes in Section \ref{sec:polytopes}. 
In Section \ref{sec:contract}, we show that the realization of $\str$ is contractible.

To emphasize that Theorem~\ref{thm:main} is not a trivial result, we use Fig.~\ref{figpre} to demonstrate that 
$data_P$ is not a convex subet of $\R^N$.
In particular, consider the vectors $v=(v_1,\ldots, v_4)$ and  $w=(w_1,\ldots, w_4)$ on the left of Fig.~\ref{figpre}.
It is left to the reader to check that $\dgm(v)=\dgm(w)$ and that this persistence diagram is given by the pair of black dots (see right of Fig.~\ref{figpre}).
Note that the vectors in $\R^4$, indicated 
(on the left) in blue and red,
lie on a straight line from $v$ to $w$.
However, the persistence diagrams indicated 
(on the right) in blue and red clearly differ from $\dgm(v)$. 
Thus, the red and blue vectors do not lie in $data_{\dgm(v)}$.

\begin{figure}[h]
\center{
\includegraphics[width=0.9\textwidth]{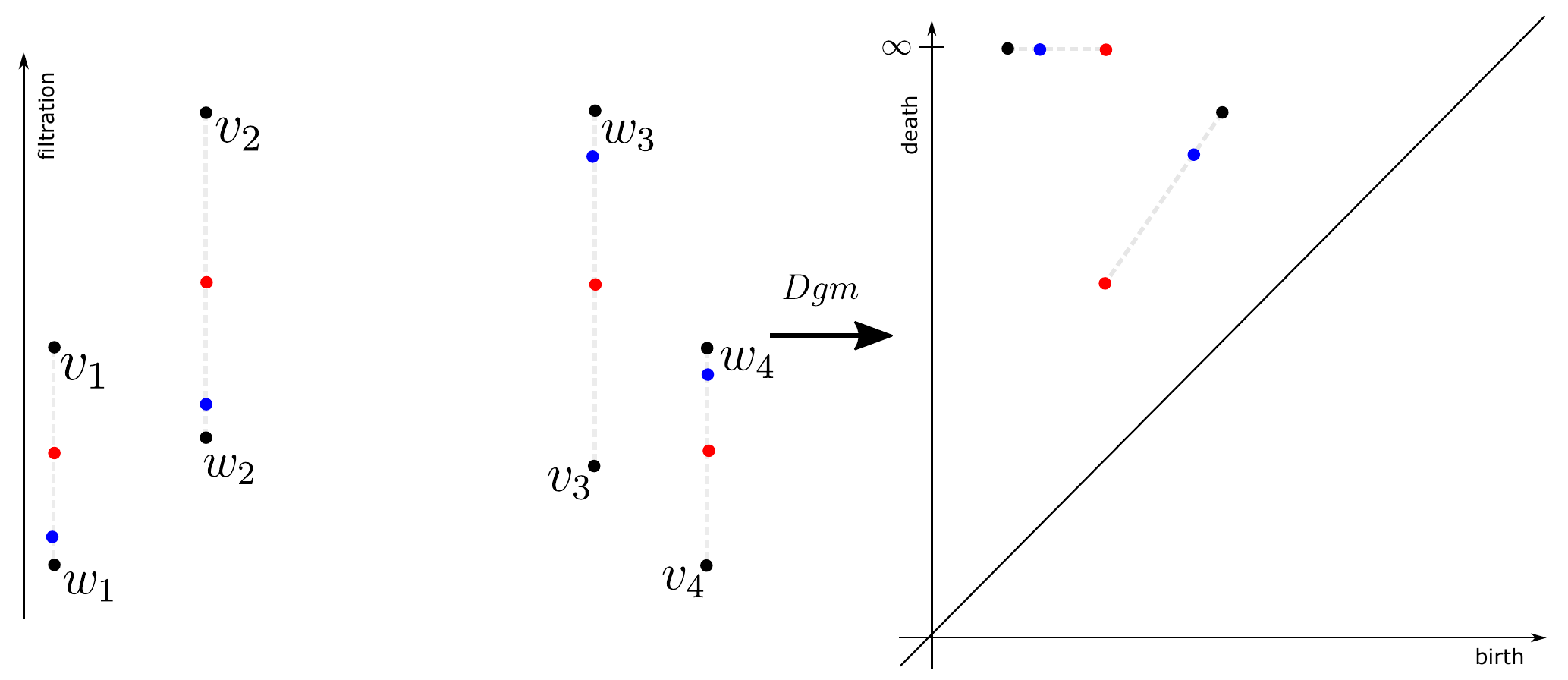}}
\caption{Non-convexity of the preimage $data_P$
under the persistence map $\dgm$. 
In the left figure the two vectors, 
$v=(v_1,...,v_4)$ and 
$w=(w_1,...,w_4)$,
lie in the preimage of the persistence diagram $P$, composed out of two black points visible on the right figure. 
Applying $\dgm$ to convex linear combinations of $v$ and $w$ results in a path in the
persistence plane illustrated on the right
(The convex path is marked in grey, and two sample vectors on the path are 
in red and blue).}
\label{figpre}
\end{figure}

In Section~\ref{sec:dynamics} we apply Theorem~\ref{thm:main} to prove the existence of fixed points with given persistence diagrams for a dissipative ordinary differential equation.

\subsection{Acknowledgements}
The work of JC and KM was partially supported by grants NSF-DMS-1125174, 1248071, 1521771 and a DARPA contracts HR0011-16-2-0033.
CW was supported by NSF grant 1702233.
In addition KM was partially supported by DARPA contract FA8750-17-C-0054, NIH grant R01 GM126555-01, and NSF grants 1934924, 1839294, and 1622401. JC was partially supported by NAWA Polish Returns grant PPN/PPO/2018/1/00029. 

The first two authors are grateful to an anonymous reviewer of a dramatically different version of this paper for suggesting the relation of our efforts to $K$-theory, which greatly simplified the proof of Theorem~\ref{thm:main}.

\section{Invariants for a fixed persistence diagram.}

Fix a persistence diagram $P$.
To describe the structure of the 
space $data_P$, we introduce
two levels of invariants:
the critical value sequences, representing the connected
components of $data_P$, and
(for each of these), a partially ordered set $\str$ indexing a polytope decomposition of 
the component.

\subsection{Components}

Fix a persistence diagram $P$.
To describe the (finitely many) connected components of $data_P$, it is useful to introduce notation that records the order in which the relevant
local maxima and minima occur.

We say that $z=(z_1,\dots,z_N)\in\R^N$ is a \emph{typical point} if its coordinates are distinct. 
If $z$ is a typical point and $1<n<N$, we say that $z_n$ is a 
\emph{local minimum} (of $z$) 
if $z_{n-1}>z_n <z_{n+1}$, and 
a \emph{local maximum} if $z_{n-1}<z_n >z_{n+1}$; it is a \emph{local extremum} if it is a local minimum or maximum.  
We say that $z_1$ and $z_N$ are \emph{boundary extrema}; $z_1$ is a local minimum (resp., maximum) if $z_1<z_2$ (resp., $z_1>z_2$). 

\begin{definition}
\label{def:cv(z)}
The \emph{critical value sequence} of a typical point $z=(z_1,\dots,z_N)$ is
\[
\cv(z) = \left(z_{n_1},\ldots, z_{n_K}\right)\in \R^K,
\] 
where the $z_{n_k}$ are the local extrema of $z$, excluding boundary extrema that are local maxima,
and $n_1 < n_2 < \ldots< n_K$.
\end{definition}

\begin{example}
\label{ex:cv(z)}
Let $z=(1.5,-0.9,1.1,2.1,1.4)\in \R^5$.
The local minimum is $z_2$ and the local maximum is $z_4$. 
The boundary extrema are $z_1$ and $z_5$.
Since $z_1$ is also a local maximum we do not include it in the critical value sequence.
Thus $h=\cv(z) = (-0.9,2.1,1.4)$. 
\end{example}

The following notion emphasizes the structure of the critical value sequences. 
\begin{definition}
\label{def:cv}
A $010$ \emph{critical value sequence} of (odd) length $K$ is a vector $\textrm{cv}=(z_1,\dots,z_K)\in\R^K$ with the property that 
 \[
 z_{n_1} < z_{n_2} > z_{n_3} < \cdots < z_{n_{K-1}} >z_{n_K}.
 \]
A $101$ \emph{critical value sequence} is defined similarly, with the inequalities reversed.
\end{definition}
Since we are using sublevel set filtrations to compute the persistence diagram we focus on $010$ critical value sequences.

Lemma \ref{lemma:local_extrema} 
below shows that the local extrema of $z$ are determined up to order by its
persistence diagram, and hence that there are only finitely many critical value sequences for any fixed persistence diagram.

Recall that a persistence diagram
is a finite collection of {\it persistence points} 
$\setof{p_i = (p_i^b, p_i^d)}$, 
where $p_i^b$ and $p_i^d$ denote birth and death values, respectively.
Since $\complex$ is connected, the persistence diagram of a typical point 
$z$ has a unique persistence point $p_i=(p_i^b,p_i^d)$ such that
$p_i^b=\min_{n=1,\ldots, N}z_n$ and $p_i^d = \infty$;
without loss of generality, we may relabel $p_i$ as $p_1$.

\begin{lemma}
\label{lemma:local_extrema}
Let $z\in\R^N$ be a typical point with persistence diagram
$\setof{p_m = (p_m^b, p_m^d) \mid m=1,\ldots,M}$. 
Then, $z$ has $K=2M-1$ local extrema; the local minima of $z$ are precisely $\setof{p_m^b}_{m=1}^M$
and the interior local maxima of $z$ are precisely $\setof{p_m^d}_{m=2}^M$.
\end{lemma}

We leave the proof of Lemma \ref{lemma:local_extrema}
to the reader, 
remarking that it
still holds when $z\in\R^N$ is not a typical point, except that the
persistence diagram may be a multiset 
(there may be multiple copies of a single persistence point).

Given a point $z$ with persistence diagram $P$,
let $C(z)$ denote the component of $data_P$ containing $z$.

The following lemma shows that $data_P$ is the disjoint union of the finitely many disjoint
components $C(z)$, indexed by the critical value sequences.  
The proof follows from the observation that the order of the local extrema cannot be changed while preserving the persistence diagram.

\begin{lemma}
\label{lem:fixComponent}
If $z$ and $z'$ are typical points in $\R^N$ then $C(z)=C(z')$ if and only if
$\cv(z) = \cv(z')$.

Moreover,
$C(z)$ is the closure of the set of typical points in $C(z)$.
\end{lemma}

\noindent This proves the first assertion in  Theorem \ref{thm:main}.

\begin{remark}
The components $C(z)$ group vectors into equivalence classes that can be characterized
using the notion of \emph{chiral merge tree} as defined in \cite{curry}. 
Corollary 5.5 of \cite{curry}] 
shows that the number of chiral merge trees realizing diagram $P$ is equal to $2^{N-1}\prod_{j=2}^N{\mu_B(I_j)}$, where $B$ is the barcode realization of $P$, i.e. set of intervals $I_j= [b_j,d_j]$ having the birth and death values of the $j$-th persistence point as its endpoints, and $\mu_B(I_j)$ is the number of intervals in $B$ that contain $I_j$. 
\end{remark}

\subsection{Cellular strings}
\label{sec:cellular}

In this section, we define the poset $\str(N,M)$ of \emph{cellular strings} associated to $M$ points arising from a vector in $\R^N$.
Thus we fix $N$ and $M$, where $N\geq 2M-1$.

Consider a string of symbols $s=s_1\cdots s_N$ of length $N$, where each symbol $s_n$ is either $0$, $1$, or $X$ (we refer to $0$ and $1$ as \emph{bits}). 
Any such string can be represented as $s=\gamma_1\cdots \gamma_J$ where each block $\gamma_j$ is a substring made up of a single symbol (that is, $\gamma_j$ is $0\cdots 0$, $1\cdots 1$, or $X\cdots X$),
and consecutive blocks have different symbols.
We refer to $s=\gamma_1\cdots \gamma_J$ as the \emph{canonical representation} of $s$.

\begin{definition}
\label{def:cellular}
Fix $M<N$. 
A $010$ \emph{cellular string}\footnote{A $101$ cellular string is defined similarly, interchanging $0$ and $1$.}
is a symbol string $s$ of length $N$ such that, 
for the canonical representation 
$s=\gamma_1\cdots \gamma_J$:
\begin{enumerate}
    \item[(i)] the symbols that make up $\gamma_j$ and $\gamma_{j+1}$ are different;
    \item[(ii)] $\gamma_1$ and $\gamma_J$ consist of the symbols $0$ or $X$;
    \item[(iii)] if $\gamma_j$ consists of the symbol $X$, then the symbol of $\gamma_{j-1}$ is different from the symbol of $\gamma_{j+1}$;
    \item[(iv)] there are exactly $M$
    values of $j$ for which $\gamma_j$ consists of the symbol $0$. 
\end{enumerate}
The set $\str(N,M)$ of cellular strings is a poset,
where $s'< s$ if the string $s$ is obtained from $s'$ by replacing some of the bits $0$ and $1$ in $s'$ by $X$.

The \emph{dimension} of a cellular string $s$,
$\dim(s)$, is the number of symbols $X$ in $s$. 
It follows from (iv) that $M$ of the blocks $\gamma_j$ have the form $0\cdots0$, and $M-1$ have the form $1\cdots1$. 
Thus, $K=2M-1$ of the blocks are bitstrings.  
If these bitstrings are $\gamma_{j_1}, \cdots, \gamma_{j_K}$,
then the symbol for $\gamma_{j_k}$ is 
$0$ if $k$ is odd and $1$ if $k$ is even. 
Since each block has at least one symbol, 
it follows that any cellular string has 
dimension at most $L=N-K$.

We write $\str^{(r)}(N,M)$ for
the sub-poset of all cellular strings 
whose first $r-1$ symbols are $X$.
Note that $\str(N,M) = \str^{(1)}(N,M)$
and $\str^{(L+1)}(N,M)=\{X\cdots X010\cdots10\}$.
\end{definition}

\begin{proposition}
An element of $\str(N,M)$ is  maximal if and only if it is an $L$-dimensional cellular string,
where $L = (N-K)$.
\end{proposition}

\begin{proof}
Let $s=\gamma_1\dots\gamma_J\in \str(N,M)$.
By definition, $\dim(s)\le L$.
Conversely, suppose that the symbol $X$ appears in $s$ has less than $L$ times.
Then some bitstring $\gamma_j$ has length $\ge2$. 
Let $s'$ be the cellular string obtained by replacing the first symbol of $\gamma_j$ by $X$.  Then $s<s'$, so $s$ is not maximal.
\end{proof}

Since both $N$ and $M$ are fixed in our analysis, we simplify the notation and write $\str$ for $\str(N,M)$.
Figure \ref{fig:K=2N=3} illustrates the poset $\str$
when $M=2$, $K=3$ and $N=5$; the right column is $\str^{(2)}$.

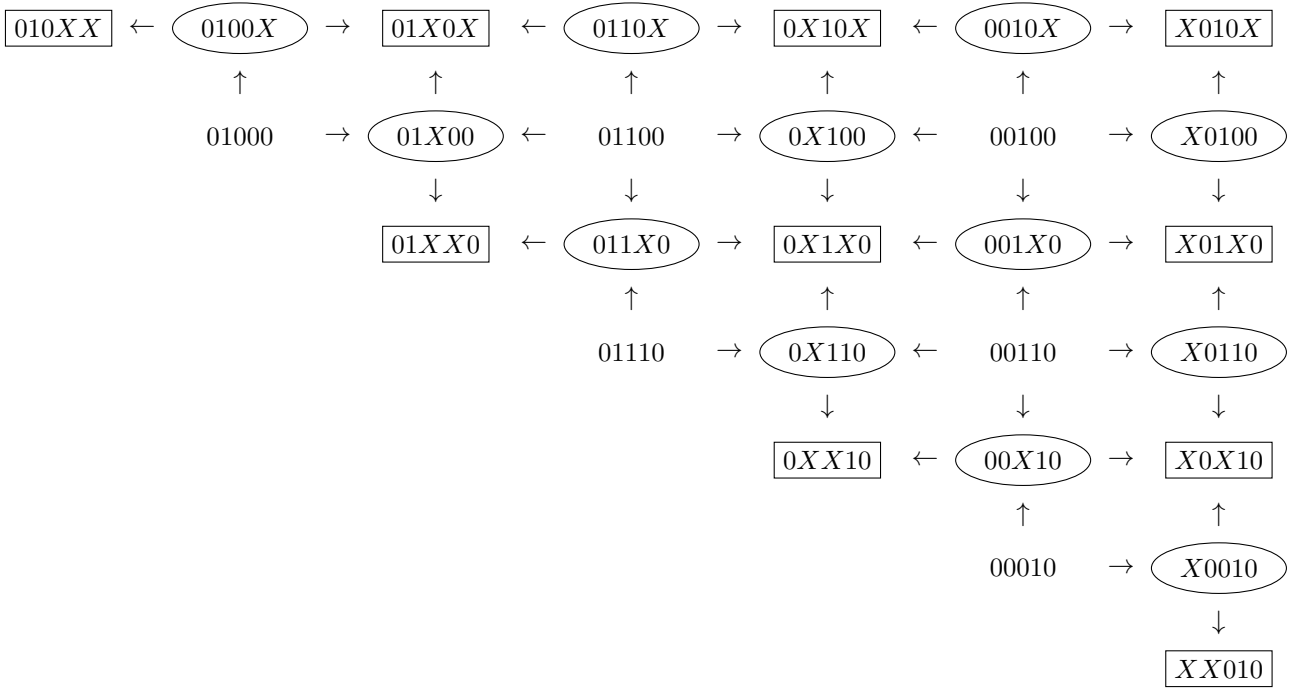
\begin{figure}
{\small
    \begin{tikzpicture}
    \matrix[nodes={draw},row sep=0.1cm, column sep=0.1cm]{
    \node[rectangle] {$010XX$}; & \node[draw=none] {$\leftarrow$}; &
    \node[ellipse] {$0100X$}; & \node[draw=none] {$\rightarrow$}; &
    \node[rectangle] {$01X0X$}; & \node[draw=none] {$\leftarrow$}; &
    \node[ellipse] {$0110X$}; & \node[draw=none] {$\rightarrow$}; &
    \node[rectangle] {$0X10X$}; & \node[draw=none] {$\leftarrow$}; &
    \node[ellipse] {$0010X$}; & \node[draw=none] {$\rightarrow$}; &
    \node[rectangle] {$X010X$}; \\
    &     &
    \node[draw=none] {$\uparrow$}; &    &
    \node[draw=none] {$\uparrow$}; &    &
    \node[draw=none] {$\uparrow$}; &    &
    \node[draw=none] {$\uparrow$}; &    &
    \node[draw=none] {$\uparrow$}; &    &
    \node[draw=none] {$\uparrow$};\\
    &    &
    \node[draw=none] {$01000$}; & \node[draw=none] {$\rightarrow$}; &
    \node[ellipse] {$01X00$}; & \node[draw=none] {$\leftarrow$}; &
    \node[draw=none] {$01100$}; & \node[draw=none] {$\rightarrow$}; &
    \node[ellipse] {$0X100$}; & \node[draw=none] {$\leftarrow$}; &
    \node[draw=none] {$00100$}; & \node[draw=none] {$\rightarrow$}; &
    \node[ellipse] {$X0100$}; \\
    &    &    &    &
    \node[draw=none] {$\downarrow$}; 
    &    &
    \node[draw=none] {$\downarrow$}; &  &
    \node[draw=none] {$\downarrow$}; &  &
    \node[draw=none] {$\downarrow$}; &  &
    \node[draw=none] {$\downarrow$};  \\
    &    &    &     &
    \node[rectangle] {$01XX0$}; & \node[draw=none] {$\leftarrow$}; &
    \node[ellipse] {$011X0$}; & \node[draw=none] {$\rightarrow$}; &
    \node[rectangle] {$0X1X0$}; & \node[draw=none] {$\leftarrow$}; &
    \node[ellipse] {$001X0$}; & \node[draw=none] {$\rightarrow$}; &
    \node[rectangle] {$X01X0$}; \\
    &    &    &    &    &    &
    \node[draw=none] {$\uparrow$}; &    &
    \node[draw=none] {$\uparrow$}; &    &
    \node[draw=none] {$\uparrow$}; &    &
    \node[draw=none] {$\uparrow$}; \\
    &    &    &    &     &    &
    \node[draw=none] {$01110$}; &  \node[draw=none] {$\rightarrow$}; &
    \node[ellipse] {$0X110$}; & \node[draw=none] {$\leftarrow$}; &
    \node[draw=none] {$00110$}; & \node[draw=none] {$\rightarrow$}; &
    \node[ellipse] {$X0110$}; \\
    & & & & & & & &
    \node[draw=none] {$\downarrow$}; & &
    \node[draw=none] {$\downarrow$}; & &
    \node[draw=none] {$\downarrow$};   \\
    & & & & & & & &
    \node[rectangle] {$0XX10$}; & \node[draw=none] {$\leftarrow$}; &
    \node[ellipse] {$00X10$}; &  \node[draw=none] {$\rightarrow$}; &
    \node[rectangle] {$X0X10$}; \\
    & & & & & & & & & &
    \node[draw=none] {$\uparrow$}; & &
    \node[draw=none] {$\uparrow$}; \\
    & & & & & & & & & & 
    \node[draw=none] {$00010$}; & \node[draw=none] {$\rightarrow$}; &
    \node[ellipse] {$X0010$}; \\
    & & & & & & & & & & & &
    \node[draw=none] {$\downarrow$}; \\
    & & & & & & & & & & & &
    \node[rectangle] {$XX010$}; \\
    };
		\end{tikzpicture}
    \caption{The string poset $\str$ for $M=2$ and $N=5$. Two-dimensional, one-dimensional, and zero-dimensional strings are surrounded by rectangles, ellipses, and nothing, respectively.  The arrows indicate the partial order. The rightmost column is the sub-poset $\str^{(2)}$.}
    \label{fig:K=2N=3} }
\end{figure}

\begin{lemma}
\label{lem:glb}
Every string $s'$ is the greatest lower bound of the set of $L$-dimensional strings $s$ with $s'<s$.
\end{lemma}

It follows that $\str$ has 
the least upper bound property: 
if two strings have a lower bound, 
they have a greatest lower bound.

\begin{proof}
We proceed by downward induction on the dimension $d$ of $s'$, the case $d=L$ being clear. 
Consider the canonical representation, $s'= \gamma_1\cdots \gamma_J$.
If $d<L$, then some bitstring $\gamma_j$ has length $\ge2$. 
Consider the strings
$s_1=\gamma_1\cdots\gamma_{j-1}X\bar{\gamma}\gamma_{j+1}\cdots \gamma_J$ and $s_2=\gamma_1\cdots\gamma_{j-1}\bar{\gamma} X\gamma_{j+1}\cdots \gamma_J$
where $\bar{\gamma}$ is a bitstring consisting of the same symbol as $\gamma_j$ but of length one less than $\gamma_j$.
Since this is the form of any cellular string $s$ satisfying $s'<s$ and
$\dim s = \dim s' +1$, the result follows.
\end{proof}

Let $s$ be an $L$-dimensional 
cellular string.
Successively replacing an $X$ adjacent to a bit (0 or 1) by that bit
yields a chain of strings $s=s_L>s_{L-1}>\cdots>s_1>s_0$.
It follows that every maximal chain in the poset has length $L$.

\begin{example}
\label{ex:I}
Consider a string 
$s(n)=\sigma_10X\cdots X1\sigma_2$
with a block of $n$ consecutive $X$'s
(where $\sigma_1$ and $\sigma_2$ are
fixed substrings).
Let $\str/s(n)$ denote the sub-poset of $\str$ consisting of 
all strings $s'\le s(n)$ which begin in $\sigma_10$ and end in $1\sigma_2$. Then $\str/s(n)$
is isomorphic to the poset $I_n$
of integer intervals $[i,j]$ with 
$1\le i\le j\le n+1$. (The string corresponding to $[i,j]$ is
\[ \sigma_10\cdots0X\cdots X1\cdots1\sigma_2;
\]
it has 
$i$ 0's and the first 1 is 
in the $(j+1)^{st}$ spot.)

If $s$ is 
a cellular string with $k$ blocks of successive $X$'s (of lengths $n_1,...,n_k$),
the sub-poset $\str/s$ of strings $s'<s$ in
$\str$ is isomorphic to the product of posets
$\str/s(n_1),...,\str/s(n_k)$, 
i.e., to the poset 
$$
I_{n_1}\times\cdots\times I_{n_k}.
$$
\end{example}

\subsection{The polytopes}
\label{sec:polytopes}

We now turn to identifying the polytopes of Theorem~\ref{thm:main}.
Fix a 010 critical value sequence $\cv = \left(z_{n_1},\ldots,z_{n_K}\right)$ 
as in Definition \ref{def:cv}.
To each $d$-dimensional cellular string $s$ we assign a $d$-dimensional polytope 
$T(s)$ in $\R^N$; $T(s)$ will be a product of simplices.  

Let $s = \gamma_1\gamma_2\cdots \gamma_J$
be the canonical representation of a
string $s$, as in 
Definition \ref{def:cellular}.
Let $n_j$ denote the length of the substring $\gamma_j$, so $N=\sum n_j$.

\begin{itemize}
\item 
If $\gamma_j$ is either $0\cdots0$ or $1\cdots1$, and $\gamma_j$ is the $k^{th}$ block from the left involving $0$ or $1$,
we set 
\[ 
T(\gamma_j) = \setof{z_k}^{n_k}
 = (z_k,...,z_k).
\]
\item
If $\gamma_1$ is a block $X\cdots X$, then
\[
T(\gamma_1) = \setof{ (x_1,\ldots,x_{n_j})\in\R^{n_j}: \infty\geq x_1\geq \cdots \geq x_{n_1} \geq z_1 }.
\]
\item 
If $\gamma_J$ is a block $X\cdots X$, then
\[
T(\gamma_J) = \setof{ (x_1,\ldots,x_{n_j})\in\R^{n_j}: z_k\leq x_1\leq \cdots \leq x_{n_1} \leq \infty }.
\]
\item 
If $\gamma_j$ is a block $X\cdots X$ 
(for $1 < j < J$),
and $\gamma_{j-1}$ is the $k^{th}$
block from the left involving $0$ or $1$,
then
\[
T(\gamma_j) = \setof{ (x_1,\ldots,x_{n_j})\in\R^{n_j}}
\quad\text{where}\quad
\begin{cases}
z_k \leq x_1 \leq \cdots \leq x_{n_j}\leq z_{k+1} & \text{if $k$ is odd;} \\
z_k \geq x_1 \geq \cdots \geq x_{n_j}\geq z_{k+1} &\text{if $k$ is even.} 
\end{cases}
\]

\item
We define $T(s)\subset\R^N$ to be
the concatenation:
\[
 T(s)  = T(\gamma_1\gamma_2\cdots \gamma_J) = \prod_{j=1}^J T(\gamma_j).
\]
\end{itemize}

Let $P$ be a persistence diagram and
$z\in dgm^{-1}(P)$. The component $C(z)$ 
of $data_P$ is the union of the T(s), 
where $s\in\str$ and $T(s)$ is defined using the critical value sequence $\cv(z)$.
This is clear from Definition \ref{def:cv(z)}.

Since the critical value sequence 
is always assumed to be fixed, 
we will suppress it in the notation.

\begin{example}
Consider the case $K=3$ and $N=5$. If $s=01XX0$, then $(\gamma_1,\ldots,\gamma_4) = (0,1,XX,0)$.
So, $(n_1,n_2,n_3,n_4)=(1,1,2,1)$ and hence
\[
T(01XX0) = \setof{z_1}\times \setof{z_2}\times \setof{(x_1,x_2) : z_2 \geq x_1 \geq x_2 \geq z_3}\times \setof{z_3}
\cong \Delta^0\times\Delta^0\times
\Delta^2\times\Delta^0.
\]
If $s=X01X0$, then $(\gamma_1,\ldots,\gamma_4,\gamma_5) = (x,0,1,x,0)$.
So, $(n_1,n_2,n_3,n_4,n_5)=(1,1,1,1,1)$ and hence
\[
T(X01X0) = [z_1,\infty)\times \setof{z_1}\times \setof{z_2}\times [z_3,z_2]\times \setof{z_3}
\cong [0,\infty)\times\Delta^0\times
\Delta^0\times\Delta^1\times\Delta^0
\]
Similarly, 
$
T(X0100) = [z_1,\infty)\times \setof{z_1}\times \setof{z_2}\times \setof{z_3}\times \setof{z_3}
\cong [0,\infty)\times\Delta^0\times
\Delta^0\times\Delta^0\times\Delta^0$.
\\
Observe that $X0100 <X01X0$ and $T(X0100)\subset T(X01X0)$.
\end{example}

Let $\mathbf{Poly}$ denote the poset of
polytopes in $\R^N$ under inclusion.
By definition, $T$ maps strings in
$\str$ to polytopes in $\mathbf{Poly}$.

\begin{lemma}
\label{lem:Poly}
$T:\str\to \mathbf{Poly}$ 
is an injective poset morphism,
and preserves greatest lower bounds.
\end{lemma}

\begin{proof}
Suppose that $s'<s$ and 
$1+\dim s' = \dim s$. 
If $s' = \gamma_1\cdots \gamma_J$ is the canonical form, then some $\gamma_j$ has the form 
$a\cdots a$ (where $a$ is 0 or 1), and $s$ has 
the form 
\[
s_1=\gamma_1\cdots \bar{\gamma}_jX \cdots \gamma_J 
\quad\textrm{or}\quad
s_2=\gamma_1\cdots X\bar{\gamma}_j \gamma_J,
\]
where $\bar{\gamma}_j=a\cdots a$ 
has one fewer bit that $\gamma_j$.
It is clear from the definition of $T$ that
$T(s_1)\ne T(s_2)$, and 
$T(s')$ is the intersection of
$T(s_1)$ and $T(s_2)$,
as desired.
\end{proof}

\subsection{Geometric Realization of Posets}

Let $C$ be a poset 
(partially ordered set).
For any $c\in C$, we write
$C/c$ for the sub-poset 
$\setof{c':c'\le c}$; 
$C$ is the union of the $C/c$.
If $c_1$ and $c_2$ have a
greatest lower bound 
$c_{12}$, then 
$(C/c_1)\cap(C/c_2)=C/c_{12}$.

By definition, the geometric realization $BC$ of any poset $C$ is a simplicial complex
whose $k$-dimensional simplices are indexed by the chains
$c_0<c_1<\cdots c_k$ 
of length $k$ in $C$. 
It is the union of
the realizations $B(C/c)$ of the sub-posets $C/c$; if
 $c_1$ and $c_2$ have a
greatest lower bound $c_{12}$,
then $B(C/c_1)$ and $B(C/c_2)$
intersect in $B(C/c_{12})$.
See \cite[IV.3.1]{K-book}
for more details.

Here are some basic facts;
see \cite[IV.3]{K-book} for
a discussion.
A poset morphism $f:C\to C'$
determines a continuous map
$BC\to BC'$, and a natural transformation
$\eta:f\Rightarrow f'$
between morphisms gives a homotopy $B\eta:BC\to BC'$
between $f$ and $f'$.
In addition, realization commutes with products: $B(C_1\times C_2)\cong (BC_1)\times(BC_2).$
Applying these considerations
to the poset $\str$, we see that
its realization $B\str$ 
is the union of the polytopes
$B(\str/s)$, and 
if $s_{12}$ is the 
greatest lower bound of $s_1$ and $s_2$
then $B(\str/s_1)\cap B(\str/s_2)$ 
is $B(\str/s_{12})$.

Let $s$ be a cellular string.
We saw in Example \ref{ex:I}
that the poset $\str/s$ is isomorphic to the product 
$I_{n_1}\times\cdots\times I_{n_k}$ 
of the posets $I_{n_j}$ of integer intervals in $[1,n_j+1]$,
corresponding to the blocks of $n_j$ succesive $X$'s in $s$. 
It is well known that $B(I_n)$ is homeomorphic to the $n$-simplex $\Delta^n$. Thus
\[
B(\str/s) \cong
\prod B(I_{n_j}) \cong
\Delta^{n_1} \times\cdots\times \Delta^{n_k}.
\]
By construction, 
$T(s)=\prod T(\gamma_j)$ also has this form.
Hence we have a natural homeomorphism
\[
B(\str/s) \cong \prod B(\str/s(n_j)) \cong
\prod B(I_{n_j}) \cong 
\prod T(\gamma_j) = T(s).
\]

\begin{theorem}
$B\str$ is homeomorphic to $C(z)$.
\end{theorem}

\begin{proof}
By construction, $C(z) = \bigcup T(s)$,
and $B\str = \bigcup B(\str/s)$.
It suffices to observe that for each 
$s_1,...,s_n$ the restriction of the
$B\str/s_i \cong T(s_i)$ induces a
homeomorphism between the
intersection of the $B(\str/s_i)$ 
and the intersection $T(s_i)$.
This holds because the two sides are
identified with $B(\str/s')$ and $T(s')$,
where $s'$ is the greatest lower bound 
of the $s_i$.
\end{proof}

\section{Contractibility}
\label{sec:contract}

We now define a poset morphism
$F_1:\str\to \str$, and modify it to define
poset morphisms $F_\ell\colon \str^{(\ell)}\to \str^{(\ell)}$ for $\ell>1$.

\begin{definition}
\label{defn:F1Lower}
Let $s$ be an $L$-dimensional cellular string.
We define $F_1(s)$ to be the string obtained from $s$ by transposing the first
(i.e., leftmost) $X$ with the bit immediately
preceding it. 
If $X$ is the initial symbol, we set $F_1(s)=s$.

If $s$ is a lower-dimensional 
cellular string, we define $F_1(s)$ as follows. 
If $s$ has an initial $X$ with no $00$ or $11$ preceding it, we do as before: transpose $X$ with the bit immediately preceding it, or do nothing if $X$ is the initial symbol.  
If $s$ begins with a block of $n+1$ zeroes, say $s=00\cdots 0\sigma_2$, we replace the initial $0$ by $X$, so  $F_1(s)=X0\cdots 0 \sigma_2$.
Otherwise, the string must have the form 
$s'=\sigma_1abb\sigma_2$,
where $a,b$ are bits, $a\ne b$, 
$\sigma_1$ is an (alternating)
bitstring not ending in $a$, and
$\sigma_2$ is the remainder of the string.
We set
\[
F_1(s')=\sigma_1aab\sigma_2. 
\]

The definition of $F_\ell\colon \str^{(\ell)}\to \str^{(\ell)}$ 
mimics that of $F_1$.
Specifically, if $s=\beta\sigma$,
where $\beta=X\cdots X$ is a
block of length $\ell-1$ then
$F_{\ell}(s) = \beta F_1(\sigma)$.
\end{definition}

\begin{example}
In Figure~\ref{fig:K=2N=3}, 
the map $F_1$ sends strings surrounded by rectangles (resp., ellipses) from one column to strings surrounded by rectangles (resp., ellipses)
in the second column to the right, while leaving the last column fixed.  Thus
$F_1(01100)= 00100$ and $F_1(00100)=X0100$.

Since $\str^{(2)}$ is the rightmost column, the map $F_2$ acts on this column, mapping
strings surrounded by rectangles (resp., ellipses) to
those two rows down. Thus
$F_2(XX010)=XX010$, $F_2(X0010)=XX010$, and
$F_2(XX010)=XX010$.
\end{example}

\begin{lemma}
\label{lemma:posetmap}
$F_1:\str\to \str$ is a poset morphism,
and is the identity on
the sub-poset $\str^{(2)}$.

Furthermore, $F_1^K(\str) = \str^{(2)}$. 
\end{lemma}

\begin{proof}
We proceed by downward induction on
$d=\dim(s)$ to show that 
if $s'<s$ then
$F_1(s')\le F_1(s)$.
If $s'$ contains an $x$ with no 
$00$ or $11$ preceeding it, the same
is true for $s$ and the inequality is
evident. 

Next, suppose that
$s'=\sigma_1abb\cdots b\sigma_2$;
either $s_1=\sigma_1aXb\cdots b\sigma_2\le s$
or else $s_2=\sigma_1ab\cdots bx\sigma_2\le s$. By induction, $F_1(s_1)\le F_1(s)$
or $F_1(s_2)\le F_1(s)$, so it suffices
to observe that $F_1(s')\le F_1(s_1),
F_1(s_2).$

Finally, if $s'=00\cdots0\sigma$ then
either $s_1=X0\cdots 0\sigma\le s$
or else $s_2=00\cdots0X\sigma\le s$.
By induction, $F_1(s_1)\le F_1(s)$
or $F_1(s_2)\le F_1(s)$, so it suffices
to observe that $F_1(s')\le F_1(s_1),
F_1(s_2).$
\end{proof}

\begin{remark}
The proof of Lemma \ref{lemma:posetmap} also shows that each $F_\ell$ is a poset morphism.
\end{remark}

We can filter the poset $\str$ 
by sub-posets $Fil_i$, where $Fil_0=\str^{(2)}$, $Fil_K=\str$
and $Fil_i$ is the full poset 
on the set of
strings $s$ with $F_1^i(s)\subset \str^{(2)}$.
In Figure \ref{fig:K=2N=3}, for example,
$Fil_1$ (resp., $Fil_2$) is the rightmost 
3 columns (resp., 5 columns).
Since $F_1$ maps $Fil_i$ to $Fil_{i-1}$, the geometric realization of $BF_1$ restricts to a continuous map from
$BFil_i$ to $BFil_{i-1}$. 
We will prove:

\begin{proposition}
The inclusions $BFil_{i-1}\subseteq BFil_i$
are homotopy equivalences.
Hence $B\str^{(2)}\subseteq B\str$
is a homotopy equivalence.
\end{proposition}

\begin{proof}
For $i>0$, we define poset morphisms 
$F_{1,i}:Fil_i\to Fil_{i-1}\subseteq Fil_i$ 
to be the identity on $Fil_{i-1}$
and $F_1$ otherwise. The geometric
realization of $F_{1,i}$ is a continuous
map $BFil_i\to BFil_{i-1}\subseteq BFil_i$
which is the identity on $BFil_{i-1}$.

We will prove that, on geometric realization, $BF_{1,i}$
is homotopic to the identity on $BFil_i$.

We define a poset morphism
$h:Fil_i\to Fil_i$ as follows.
If $s\in Fil_{i-1}$ then $h(s)=s$;
if $s\not\in Fil_{i-1}$, define $h(s)$ 
to be the greatest lower bound of
$s$ and $F_1(s)$. Thus $Bh$ is a continuous map from $BFil_i$ to itself.
For $s\in Fil_i$, the inequalities 
$s \ge h(s) \le F_{1,i}(s)$
yield natural transformations
$\textrm{id}_i {\Leftarrow} h \Rightarrow F_1$.
and hence homotopies between the maps $\textrm{id}_i$
(the identity map on $BFil_i$), $Bh$ and $BF_{1,i}$.
\end{proof}

\begin{corollary}
Each $B\str^{(\ell+1)}\subset B\str^{(\ell)}$ is a homotopy equivalence.  In particular,
the inclusion of the point
$B\str^{(L+1)}$ in $B\str$ is a
homotopy equivalence, i.e.,
$B\str$ is contractible.
\end{corollary}

\begin{remark}
We can describe the map
$T(s)\to T(F_1(s))$ induced by
$F_1$. For example, suppose that
$s=\sigma_1\gamma_{j-1}\gamma_j\sigma_2$,
where $\sigma_1=\gamma_1\cdots\gamma_{j-1}$ is an alternating bitstring of length $\ge2$ and
$\gamma_j$ is a block 
$X\cdots X$. Then $T(\gamma_{j-1}) = \{z_{j-1}\}$ and
$T(\gamma_j)\subset\R^{n_j}$
is defined by inequalities,
either $z_{j-1}\le x_1\cdots$ or
$z_{j-1}\ge x_1\cdots$, depending on the parity of $j$.
The map $F_{1}$ sends 
$T(\gamma_{j-1})\times T(\gamma_j)$ 
to the subset 
$$
T(X)\times \{z_{j-1}\}\times T(\gamma'),
$$
where $T(X)$ is defined by 
$z_{j-2}\le x_1\le z_{j-1}$
and $T(\gamma')$ is defined by the equations $z_{j-1}\le x_2\cdots$ or
$z_{j-2}\ge x_1\cdots$.
In effect, the map sends $x_1$
to $z_{j-1}.$
\end{remark}

\section{Existence of fixed points for flows}
\label{sec:dynamics}

As an application of Theorem~\ref{thm:main}, we establish the existence of a fixed point solution of a ordinary differential equation
whose trajectories are being observed in the space of persistence diagrams.
To be more precise consider a differential equation $\dot{z} = f(z)$, $z\in \R^N$, with the property that it possesses a compact global attractor $\cA$ \cite{raugel}.
Given an initial condition $z(0)=\bar{z}\in \R^N$, we write $z(t)=\varphi(t,\bar{z})$, $t\in [0,\infty)$ for the solution in forward time. 
The important consequence of the existence of a compact global attractor is that there exists $R>0$ such that for any initial condition $\bar{z}$ there exists $t_{\bar{z}}>0$ such that $\| \varphi(t,\bar{z}) \| < R$ for all $t\geq t_{\bar{z}}$.
Observing  the persistence diagrams along a trajectory results in a curve $\dgm(\varphi(t,\bar{z}))\in\per$.
In what follows we do not assume that we have knowledge of the
nonlinearity of $f$, or of the actual trajectories $\varphi(t,z)$; 
we are only
given the curves $\dgm(\varphi(t,\bar{z}))$ of persistence diagrams.

Even if the persistence diagram is constant, we cannot conclude that
the underlying differential equation has a fixed point.
As an example, consider a differential equation in $\R^3$ with a periodic solution in which the first coordinate $z_1=0$ is constant, and $(z_2,z_3)$ oscillates with the property that $1\leq z_2 \leq z_3$.
The associated curve in $\per$ consists of the constant persistence diagram $P=\setof{(0,\infty)}$.

However, Theorem~\ref{thm:maindyn} provides a scenario under which the observation of sufficiently many trajectories suggests the existence of a fixed point for the unknown ordinary differential equation that generates the dynamics.
More general theorems are possible and, as will be discussed in a later paper, these techniques can be lifted to the setting of partial differential equations defined on bounded intervals.
The purpose of this example is to emphasize the importance of Theorem~\ref{thm:main} from the perspective of data analysis.
Thus, we focus on a much more modest result.
We will show that if a particular type of neighborhood in $\per$ is positively invariant under the dynamics, i.e.\ if $\dgm(z)$ is in the neighborhood implies that $\dgm(\varphi(t,z))$ is in the neighborhood for all $t>0$, then there exists a fixed point for the differential equation that generates the dynamics.
To state and obtain such a result requires the introduction of additional notation,

\begin{definition}
A persistence diagram $P = \setof{p_m=(p_m^b,p_m^d): m = 1,\ldots, M}$ is \emph{sparse} if each persistence point is unique, i.e. $p_m \neq p_n$ for all $m\neq n$.
\end{definition}

Given a sparse persistence diagram we can choose $\mu >0$ such that $\| p_m-p_n\|_\infty \geq 4\mu$ for all $m\neq n$ and $|p_m^d - p_m^b| \geq 4\mu$ for all $m$.

\begin{example}
A sparse persistence diagram $Q$ is shown in Figure~\ref{fig:sparse}.
We can choose $\mu = 0.25$.
A possible critical value sequence associated to $Q$ is $\cv(z)=(3,4.5,1,3.5,2)$.
\end{example}

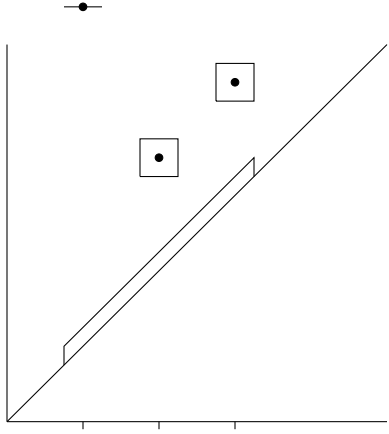
\begin{figure}
\begin{center}
    \begin{tikzpicture}
    \draw (0,0) -- (5,0);
    \draw (0,0) -- (0,5);
    \draw (0,0) -- (5,5);
    
    \draw (1,0) -- (1,-0.1);
    \draw [fill] (1,5.5) circle [radius=0.05];
    
    \draw (2,0) -- (2,-0.1);
    \draw [fill] (2,3.5) circle [radius=0.05];
    
    \draw (3,0) -- (3,-0.1);
    \draw [fill] (3,4.5) circle [radius=0.05];
    
    \draw (0.75,5.5)--(1.25,5.5);
    \draw (1.75,3.25)--(1.75,3.75)--(2.25,3.75)--(2.25,3.25)--(1.75,3.25);
    \draw (2.75,4.25)--(2.75,4.75)--(3.25,4.75)--(3.25,4.25)--(2.75,4.25);

    \draw (0.75,0.75)--(0.75,1)--(3.25,3.5)--(3.25,3.25);
    
	\end{tikzpicture}
    \caption{A sparse persistence diagram $Q$ with persistence points $\setof{(1,\infty),(2,3.5),(3,4.5)}$.  The boxes indicate the set $\sN_Q$ for $\mu = 0.25$.}
    \label{fig:sparse}
    \end{center}
\end{figure}

We use $\mu$ to define subsets of $\R^N$ and $\per$. 
We begin by constructing a subset of $\R^N$ using the set of cellular strings $\str(N,M)$. 
Choose a point $\hat{z}$
with persistence diagram $P$.
This gives rise to a fixed critical value sequence $\cv(\hat{z})$ and the associated component $C(\hat{z})\subset \R^N$ of $data_P$ is given by 
\[
C(\hat{z}) = \bigcup_{s\in \str(N,M)} T(s).
\]
By Theorem~\ref{thm:main}, $C(\hat{z})$ is a contractible union of polytopes.

Let $B_\mu(C(\hat{z}))\subset \R^N$ be the set of points that lie within a distance $\mu$ of $C(\hat{z})$ using the $\sup$-norm.
The bound on the choice of $\mu$ guarantees that if $s',s''\in \str(N,M)$ are of maximal dimension  and there does not exist $s\in \str(N,M)$ such that $s<s'$ and $s<s''$, then
$B_\mu(T(s'))$ and $B_\mu(T(s''))$
are disjoint. 
Therefore $B_\mu(C(\hat{z}))$ is contractible.

We now turn to the subset of $\per$.
For each $m=1,\ldots, M$ set 
\[
\sP_m := \setof{p=(p^b,p^d) : \|p-p_m\|_1 \leq \mu}
\]
and
\[
\sD := \setof{p=(p^b,p^d) : p^b \in [p_1^b-\mu, \sup\setof{p_m^b}+\mu]\ \text{and}\  0\leq p^d-p^b\leq \mu}.
\]
See Figure~\ref{fig:sparse}.
Define $\sN_P \subset \per$ to be the set of persistence diagrams generated by elements of $\R^N$ with the property that for each $m=1,\ldots, M$ there exists a unique persistence point in $\sP_m$ and any other persistence points lie in $\sD$.

These constructions allow us to prove the following theorem concerning the existence of fixed points of the unknown, underlying dynamical system $\varphi$.

\begin{theorem}
\label{thm:maindyn}
Consider a dynamical system generated by an ordinary differential equation that has a global compact attractor and whose trajectories are represented by $\varphi(t,z)$.
Let $P$ be a sparse persistence diagram and let $\sN_P$ be defined as above.
Assume that if $\dgm(\varphi(t_0,z))\in \sN_P$, then $\dgm(\varphi(t,z))\in \sN_P$ for all $t\geq t_0$.
Then, for each component of $\dgm^{-1}(\sN_P)\subset \R^N$ there exists a vector $\hat{z}$ such that $\dgm(\hat{z})\in\sN_P$ and $\varphi(t,\hat{z})=\hat{z}$ for all $t\in\R$, i.e.\ $\hat{z}$ is a fixed point for the dynamical system.
\end{theorem}

\begin{proof}
We begin with the observation that if $z\in B_\mu(C(\hat{z}))$ and there exists $t_1 >0$ such that $\varphi(t_1,z)\not\in B_\mu(C(\hat{z}))$, then there exists $t_0 \in (0,t_1]$ such that $\dgm(\varphi(t_0,z))\not\in \sN_P$. 
This follows from the stability theorem of persistent homology using the bottleneck distance \cite{stability}.
This contradicts the hypothesis, therefore, that $B_\mu(C(\hat{z}))$ is a contractible, positively invariant region under the dynamics.
By \cite[Proposition 3.1]{mccord:mischaikow}, the Conley index of the maximal invariant set is that of a hyperbolic attracting fixed point. 
By \cite[Corollary 5.8]{mccord} 
(which utilizes the well known Lefschetz fixed point theorem), 
the maximal invariant set in $B_\mu(C(\hat{z}))$ contains a fixed point.
\end{proof}

\section{Conclusion and Future Work}
\label{sec:conclusion}

To the best of our knowledge, this paper provides the first detailed analysis of the topology of the preimage of a persistence map.
Although we have presented the results in the context of sublevel set filtrations, the same arguments can be applied in the setting of superlevel set filtrations.
The only significant change is that one needs to use $101$ cellular strings; see Definitions~\ref{def:cv} and \ref{def:cellular}.

Theorem~\ref{thm:maindyn}, and the use of persistence diagrams to obtain results about the dynamics of an ODE, may appear somewhat artificial. 
However, consider a PDE, such as a reaction diffusion equation, defined on an interval.
A finite spatial sampling of the solution at a time point gives rise to a vector.
We can think of this vector as arising from two different proceedures: (i) numerical, e.g.\ the values of an ODE derived from a Galerkin approximation to the PDE, or (ii) experimental, e.g.\ a pixelated image of the solution.
Theorem~\ref{thm:maindyn} is applicable in both cases, and one expects that for fine enough discretization or resolution that the results of Theorem~\ref{thm:maindyn} will be applicable to the PDE.
The example involving images brings us much closer to current treatments of complex spatio-temporal dynamics \cite{kramar:levanger:tithof:2015,levanger:xu:cyranka:2019}.
Hence, the natural next step in our research is to obtain an analogous result about existence of fixed points for one-dimensional PDEs whose trajectories are observed in the persistence space.

Finally, the obvious open question as a result of this paper is: 
given a $d$-dimensional simplicial complex $\mathcal{S}$ with a function $f$, similar in form to that of Definition~\ref{def:Sfiltration}, can one determine the homology of components of the pre-image of a persistence diagram?

\bibliography{persistence_fp}
\bibliographystyle{abbrv}

\end{document}

%% file: correctm.tex
\def\corcommstyle{\bf\small\tt}

\def\corrl #1<<#2||#3>>{
\if\visiblecomments y
  \begin{quote} {\corcommstyle $<<$COMMENT$>>$ {\color{red}#1\marginpar{!!}}\\$<<$OLD$<<$} \end{quote}

{\color{red} 
 #2
 }

  \begin{quote} {\corcommstyle ==NEW== } \end{quote}
   \noindent\hrulefill
 
\vspace{-10pt} 
 
 \noindent\hrulefill
 
 \vspace{-10pt} 
 
 \noindent\dotfill
 
  #3
  
   \noindent\dotfill 

\vspace{-10pt} 
 
 \noindent\hrulefill
 
 \vspace{-10pt} 
 
 \noindent\hrulefill
  \begin{quote} {\corcommstyle $>>$END$>>$ } \end{quote}
 \else
  #3
 \fi
}

\long\def\longcorrl #1<<#2||#3>>{
\if\visiblecomments y
  \begin{quote} {\corcommstyle $<<$COMMENT$>>$ {\color{red}#1\marginpar{!!}}\\$<<$OLD$<<$} \end{quote}
 
 {\color{red}

  #2
  
  }
  
  \begin{quote} {\corcommstyle ==NEW== } \end{quote}
  
    \noindent\hrulefill
 
\vspace{-10pt} 
 
 \noindent\hrulefill
 
 \vspace{-10pt} 
 
 \noindent\dotfill
 
  #3
  
   \noindent\dotfill 

\vspace{-10pt} 
 
 \noindent\hrulefill
 
 \vspace{-10pt} 
 
 \noindent\hrulefill
  \begin{quote} {\corcommstyle $>>$END$>>$ } \end{quote}
 \else
  #3
 \fi
}

\def\mlabel #1
{
  \if\visiblecomments y
     \marginpar[\flushright \bf \footnotesize #1]{\bf \footnotesize #1}
  \fi
}

\def\flabel #1
{
  \if\visiblecomments y
       \hbox{\bf\footnotesize #1}
  \fi
}

\def\corrq #1<<#2>>{
\if\visiblecomments y
  \begin{quote} {\corcommstyle $<<$COMMENT$>>$ #1\marginpar{!!}\\$<<$BEG$<<$} \end{quote}
  \noindent\hrulefill
 
\vspace{-10pt} 
 
 \noindent\hrulefill
 
 \vspace{-10pt} 
 
 \noindent\dotfill
 
 {\color{red}
  
  #2
  
  }
   
  \noindent\dotfill 

\vspace{-10pt} 
 
 \noindent\hrulefill
 
 \vspace{-10pt} 
 
 \noindent\hrulefill 
  \begin{quote} {\corcommstyle $>>$END$>>$ } \end{quote}
 \else
  #2
 \fi
}

\long\def\longcorrq #1<<#2>>{
\if\visiblecomments y
  \begin{quote} {\corcommstyle $<<$COMMENT$>>$ #1\marginpar{!!}\\$<<$BEG$<<$} \end{quote}
  \noindent\hrulefill
 
\vspace{-10pt} 
 
 \noindent\hrulefill
 
 \vspace{-10pt} 
 
 \noindent\dotfill

  #2

  \noindent\dotfill 

\vspace{-10pt} 
 
 \noindent\hrulefill
 
 \vspace{-10pt} 
 
 \noindent\hrulefill 
  \begin{quote} {\corcommstyle $>>$END$>>$ } \end{quote}
 \else
  #2
 \fi
}

\def\corrc #1<<>>{
\if\visiblecomments y
  \begin{quote} {\corcommstyle $<<$COMMENT$>>$ \color{red} #1\marginpar{!!}} \end{quote}
\fi
}

\def\corre #1<<#2||#3>>{
\if\visiblecomments y
  #3\marginpar{\corcommstyle #1}
 \else
  #3
 \fi
}

\long\def\longcorre #1<<#2||#3>>{
\if\visiblecomments y
  #3\marginpar{\corcommstyle #1}
 \else
  #3
 \fi
}

\def\corrs #1<<#2||#3>>{
\if\visiblecomments y
  #3\marginpar{\corcommstyle #2 $\rightarrow$ #3\\ #1}
 \else
  #3
 \fi
}

\def\corro #1<<#2||#3>>{
#2}

\def\corrn #1<<#2||#3>>{
#3}

\long\def\longcorro #1<<#2||#3>>{
#2}

\long\def\longcorrn #1<<#2||#3>>{
#3}

\long\def\underconstruction #1<<<#2>>>{
\if\visiblecomments y
  \begin{quote} {\corcommstyle $<<$UNDER CONSTRUCTION - BEGIN$>>$ #1\marginpar{!!}} \end{quote}
   {\color{red}
  #2
  }
  \begin{quote} {\corcommstyle $>>$UNDER CONSTRUCTION - END$>>$ } \end{quote}
 \else
 \fi
}

\def\showcomments{
  \let\visiblecomments y
}

\def\hidecomments{
  \let\visiblecomments n
}